\documentclass[11pt]{amsart}
\usepackage{fullpage,amsmath,amsthm,amsfonts,amssymb,
graphicx,amscd,float,hyperref,enumitem,setspace}
\usepackage{comment}
\usepackage{graphicx,xypic}

\newtheorem{thm}{Theorem}[section]
\newtheorem{lem}[thm]{Lemma}
\newtheorem{qst}[thm]{Question}
\newtheorem{prop}[thm]{Proposition}
\newtheorem{rk}[thm]{Remark}

\theoremstyle{definition}
\newtheorem{df}[thm]{Definition}

  \newcommand{\NN}{\mathbb{N}}

\newcommand{\vphi}{\varphi}
\newcommand{\veps}{\varepsilon}
\newcommand{\ol}{\overline}
\newcommand{\from}{\colon}

\newcommand{\out}{\textup{Out}(F_r)}

\newcommand{\mA}{\mathcal{A}}

\newcommand{\mT}{\mathcal{T}}

\newcommand{\mS}{\mathcal{S}}

\newcommand{\G}{\Gamma}

\newcommand{\os}{CV_r}

\begin{document}

\title[Counting conjugacy classes of fully irreducibles]{Counting conjugacy classes of fully irreducibles: double exponential growth}

\author{Ilya Kapovich and Catherine Pfaff}

\address{\tt  Department of Mathematics and Statistics, Hunter College of CUNY\newline
  \indent 695 Park Avenue, New York, NY 10065
  \newline \indent  {\url{http://math.hunter.cuny.edu/ilyakapo/}}, }
  \email{\tt ik535@hunter.cuny.edu}

\address{Department of Math \& Stats, Queen's University \newline
  \indent
Jeffery Hall, 48 University Ave.,
Kingston, ON Canada, K7L 3N6
  \newline
  \indent  {\url{http://mast.queensu.ca/~cpfaff}}, }
  \email{\tt c.pfaff@queensu.ca}



\begin{abstract}
Inspired by results of Eskin and Mirzakhani~\cite{EM11} counting closed geodesics of length $\le L$ in the moduli space of a fixed closed surface, we consider a similar question in the $\out$ setting. The Eskin-Mirzakhani result can be equivalently stated in terms of counting the number of conjugacy classes (within the mapping class group) of pseudo-Anosovs whose dilatations have natural logarithm $\le L$. Let $\mathfrak N_r(L)$ denote the number of $\out$-conjugacy classes of fully irreducibles satisfying that the natural logarithm of their dilatation is $\le L$.
We prove for $r\ge 3$ that as $L\to\infty$, the number $\mathfrak N_r(L)$ has double exponential (in $L$) lower and upper bounds.
These bounds reveal behavior not present in the surface setting or in classical hyperbolic dynamical systems.
\end{abstract}

\maketitle

\section{Introduction}

The theme of counting closed geodesics plays an important role in geometry and dynamics, and primarily dates back to the seminal work of Margulis in the 1960s-1970s~\cite{Mar69,Mar70}.
Margulis considered the situation where $M$ is a closed Riemannian manifold of curvature $\le -1$, and proved that if $\mathfrak N(L)$ is the number of closed geodesics in $M$ of length $\le L$, then
\[
\mathfrak N(L)\sim \frac{e^{hL}}{hL},
\]
where $h$ is the topological entropy of the geodesic flow on $M$ (equivalently the volume entropy of $M$). Here $\sim$ means that the ratio of two functions converges to $1$ as $L\to\infty$. This result can also be interpreted as counting the number of conjugacy classes $[\gamma]$ of elements $\gamma\in \pi_1(M)$ with translation length $\le L$ in $\widetilde M$. There were many generalizations of Margulis' result to other contexts, including manifolds with cusps, manifolds of nonpositive curvature, orbifolds, and Teichm\"uller space.

While the moduli space of a surface is not a manifold, recent important work of Eskin and Mirzakhani~\cite{EM11} provides an analog of Margulis' theorem in the moduli space setting. Namely, let $\Sigma_g$ be a closed connected oriented surface of genus $g\ge 2$, let $\mathcal T(\Sigma_g)$ denote the Teichm\"uller space of $\Sigma_g$, endowed with the Teichm\"uller metric, and let $\mathcal M(\Sigma_g)$ denote the moduli space, locally also equipped with the Teichm\"uller metric. Denote by $\mathfrak N_g(L)$ the number of closed Teichm\"uller geodesics in $\mathcal M(\Sigma_g)$ of length $\le L$ and let $h=6g-6$. Eskin and Mirzakhani proved~\cite{EM11} that
\[
\mathfrak N_g(L)\sim \frac{e^{hL}}{hL}.
\]
Every closed geodesic in $\mathcal M(\Sigma_g)$ uniquely corresponds to the $MCG(\Sigma_g)$-conjugacy class of a pseudo-Anosov element $\vphi$ of the mapping class group $MCG(\Sigma_g)$ and the length of that closed geodesic is equal to the translation length of $\vphi$ along the (unique) Teichm\"uller geodesic axis $A_\vphi\subseteq \mathcal T(\Sigma_g)$. That translation length, in turn, is equal to $\log\lambda(\vphi)$, where $\lambda(\vphi)$ is the \emph{dilatation} or \emph{stretch factor} of $\vphi$. Thus, $\mathfrak N_g(L)$ is equal to the number of $MCG(\Sigma_g)$-conjugacy classes of pseudo-Anosov elements $\vphi\in MCG(\Sigma_g)$ with $\log \lambda(\vphi)\le L$.

Shifting to homotopy equivalences of graphs and outer automorphisms of free groups, we look to count $\out$ conjugacy classes. There the \emph{Culler-Vogtmann Outer space} $CV_r$ provides a counterpart to the Teichm\"uller space of a closed surface and the quotient $\mathcal{M}_r$ of $CV_r$ under the $\out$ action provides a counterpart to the moduli space $\mathcal{M}(\Sigma_g)$. There is also a natural asymmetric geodesic metric $d$ on $CV_r$ (see \cite{FrancavigliaMartino}). In the $\out$ setting, the main analog of a pseudo-Anosov mapping class is a ``fully irreducible" outer automorphism. An outer automorphism $\vphi$ is called \emph{fully irreducible} if no positive power fixes the conjugacy class of a nontrivial proper free factor.

We again have a notion of a stretch factor, relating to the translation distance along a geodesic. Let $X$ be a free basis of $F_r$. For a fully irreducible $\vphi\in\out$ and $1\ne w\in F_r$ such that $w$ does not represent a $\vphi$-periodic conjugacy class in $F_r$, the limit
\[
\lambda(\vphi):=\lim_{n\to\infty}\sqrt[n]{||\vphi^n(w)||_X}
\]
exists and is independent of $w$ and $X$. This limit is called the \emph{stretch factor} of $\vphi$. Every fully irreducible $\vphi\in \out$ admits an invariant geodesic axis in $CV_r$ on which $\vphi$ acts with translation length $\log\lambda(\vphi)$. Crucially, unlike in the Teichm\"uller space setting, such an axis is in general highly non-unique.

Our main result is:

\begin{thm}\label{thm:main}
For each integer $r\ge 3$, there exist constants $a=a(r)>1,b=b(r)>1, c=c(r)>1$ so that: For $L\ge 1$, let $\mathfrak N_r(L)$ denote the number of $\out$-conjugacy classes of fully irreducibles $\vphi\in \out$ with $\log\lambda(\vphi)\le L$. Then there exists an $L_0\ge 1$  such that for all $L\ge L_0$ we have
\[
c^{e^L}\le \mathfrak N_r(L)\le  a^{b^L},
\]
where $e$ is the base of the natural logarithm.

Therefore, $c^{e^L}$ bounds from below the number of equivalence classes of closed geodesics in $\mathcal{M}_r$ of length bounded above by $L$.
\end{thm}

By a closed geodesic in $\mathcal{M}_r$ we mean the image in $\mathcal{M}_r$ of a periodic geodesic (see \S \ref{ss:StallingsFoldDecompositions}) in $CV_r$. Two such closed geodesics in $\mathcal{M}_r$ are considered equivalent if they come from periodic geodesics in $CV_r$ corresponding to conjugate elements of $\out$. Not all closed geodesics in $\mathcal{M}_r$ come from axes of fully irreducibles, since there exist nonirreducible elements in $\out$ that admit periodic geodesic lines in $CV_r$. However, it is known, by a combination of results of Bestvina and Feighn~\cite{bf11} (see also \cite{a11}) and of Dowdall and Taylor~\cite{dt15} that a $\vphi$-periodic geodesic $A_\vphi\subseteq CV_r$ is ``contracting" with respect to the asymmetric Lipschitz metric $d$ on $CV_r$ if and only if $\vphi\in \out$ is fully irreducible. Therefore, Theorem~\ref{thm:main} can be interpreted as providing double exponential lower and upper bounds on the number of equivalence classes of ``contracting" closed geodesics of length $\le L$ in $\mathcal{M}_r$.

The case of rank $r=2$ is special, and Theorem~\ref{thm:main} does not apply. The group $Out(F_2)$ is commensurable with the mapping class group $MCG(\Sigma_{1,1})$ of the punctured torus $\Sigma_{1,1}$ and with the group $SL(2,\mathbb Z)$. The Teichm\"uller space $\mathcal T(\Sigma_{1,1})$ is the hyperbolic plane $\mathbb H^2$, with a faithful discrete isometric action of  $MCG(\Sigma_{1,1})$ as a nonuniform lattice. Counting $\mathfrak N_2(L)$ amounts to (up to correcting for the commensurability effects) computing the number of conjugacy classes in that lattice of translation length $\le L$. By the classical counting results, this produces exponential asymptotics for $\mathfrak N_2(L)$, rather than the double exponential asymptotics displayed in Theorem~\ref{thm:main}.

Counting the conjugacy classes of fully irreducibles $\vphi\in \out$ with a given bound on $\log\lambda(\vphi)$ is a considerably different problem from the corresponding mapping class group problem. The main difference is that the Outer space $CV_r$ does not admit any of the nice local analytic structures present in the Teichm\"uller space setting, precluding the use of classical methods of ergodic theory in analyzing the geodesic flow dynamics. Indeed, we will see below that the counting results we obtain in the $\out$ setting exhibit new behavior, not present in classical hyperbolic dynamical systems. Recall that for a fully irreducible $\vphi\in\out$, the number $\log\lambda(\vphi)$ can be interpreted as the length of a closed geodesic in the moduli space  $\mathcal{M}_r$ corresponding to $\vphi$. Thus, the question of counting curves can be translated into a question about stretch factors. 

As a companion result to Theorem~\ref{thm:main} we prove that the number of \emph{distinct values} of $\log\lambda(\vphi)$ in the setting of Theorem~\ref{thm:main} grows at most exponentially in $L$:

\begin{thm}\label{thm:main2}
Let $r\ge 2$ be an integer.
Then there exists a constant $p=p(r)>1$ such that
\[
\#\{ \lambda(\vphi) \mid \vphi\in\out \text{ is fully irreducible with } \log\lambda(\vphi)\le L \} \le p^L.
\]
\end{thm}
This fact stands in sharp contrast with the double exponential growth established in Theorem~\ref{thm:main}. Other $\out$-conjugacy class invariants of fully irreducibles, such as the index, the index list, and the ideal Whitehead graph of $\vphi$ (see \cite{hm11}), only admit finitely many values for a given rank $r$ and thus are also not suitable for counting conjugacy classes.

\subsection{Implied complications to defining a Patterson-Sullivan measure on $\partial CV_r$}

An interesting question raised by Theorem~\ref{thm:main} concerns trying to define a version of the Patterson-Sullivan measure on the boundary $\partial CV_r$ of the Outer space $CV_r$. Usually, in the context of a group $G$ acting isometrically on a space $X$ with some ``hyperbolic-type" features, one first defines the ``critical exponent" or ``volume entropy" $h(X)$ for the action as
\[h(X)=\lim_{L\to\infty} \frac{1}{L} \log\#\{g\in G| d_X(x_0,gx_0)\le L\},\]
where $x_0\in X$ is some basepoint. The \emph{Poincare series} for the action is $\Pi_{x_0}(s)=\sum_{g\in G} e^{-s d_X(x_0,gx_0)}$. Typically the Poincar\'e series diverges for $s=h(X)$. A \emph{Patterson-Sullivan measure} $\mu_X$ for $X$ is defined as a weak limit, as $s\searrow h(X)$, of
\[
\frac{1}{\Pi_{x_0}(s)} \sum_{g\in G} e^{-s d_X(x_0,gx_0)} Dirac(g x_0),\tag{$\ddag$}
\]
where the limit is taken with respect to a suitable compactification $\overline{X}=X\cup \partial X$ of $X$. The divergence of the Poincare series at $s=h(X)$ implies that $\mu$ is supported on $\partial X$. Patterson-Sullivan measures proved to be a useful and important tool in differential geometry and dynamics (see the survey~\cite{DeSt} for details), including the Teichm\"uller space context~\cite{Ge14}.
Although the Outer space $CV_r$ has many hyperbolic-like features, Theorem~\ref{thm:main}  shows that even counting conjugacy classes of elements in $\out$ with translation length $\le L$ in $CV_r$ already results in double exponential growth. As we discuss in more detail in \S \ref{sec:Q} below, doubly exponentially growing sequences have (at least) two distinct associated ``entropy" or ``critical exponent" quantities. Thus, in light of Theorem~\ref{thm:main}, when trying to construct a Patterson-Sullivan measure on $\partial CV_r$, one would have to somehow modify the notion of a Poincare series and then adapt $(\ddag)$ above to account for the double exponential growth and for the presence of these two entropies.

For various flows associated with hyperbolic-like dynamical systems (such as geodesic flows on negatively curved manifolds, the Teichm\"uller geodesic flow, etc), one can usually compute the topological entropy of the flow as the exponential growth rate of the number of closed periodic orbits of the flow; see, in particular, \cite[Theorems 18.5.5 and 18.5.7]{KH95}. This fact plays a key role in the proof of most of the results related to counting closed geodesics, from Margulis~\cite{Mar69,Mar70} to Eskin-Mirzakhani~\cite{EM11}. In the context of the Outer space $CV_r$, although there is no well-behaved tangent bundle, one can still consider various spaces of ``lines" in $CV_r$, such as geodesic folding lines, sometimes with some extra restrictions; see, for example,~\cite{hm11,bf11,FrancavigliaMartino,akp14}. Any such ``space of lines" $\mathcal L_r(CV_r)$ in $CV_r$ has a natural $\out$-action by translations and a natural parameter shift action of $\mathbb R$. Factoring out the action of $\out$ produces the corresponding ``space of lines" $\mathcal L_r(\mathcal M_r)$ in the moduli space $\mathcal M_r=CV_r/\out$, endowed with the $\mathbb R$-shift action. Theorem~\ref{thm:main} indicates that for most reasonable choices of $\mathcal L_r(CV_r)$, the space   $\mathcal L_r(\mathcal M_r)$ will have infinite topological entropy with respect to this $\mathbb R$-action. Therefore, studying $\mathcal L_r(\mathcal M_r)$ as a dynamical system, e.g. trying to understand shift-invariant measures on $\mathcal L_r(\mathcal M_r)$,  will require properly accounting for the double exponential growth exhibited in Theorem~\ref{thm:main}.

\subsection{Proof complications, ideas, and sketch}

We discuss the main ideas in the proof of Theorem~\ref{thm:main} here. However, we first remark that the original proof by Margulis in the manifold setting, as well as the proofs of most generalizations, exploit the properties of the geodesic flow on the underlying structure and ultimately rely on some form of coding by symbolic dynamical systems. We not only lack such a flow but lack a local analytic structure. Further, as a result of a fully irreducible outer automorphism having many associated axes, it is easy to over-count. We thus heavily rely on the lone axis machinery developed in \cite{LoneAxes16}.

Every fully irreducible $\vphi\in \out$ is the induced map of fundamental groups for an ``efficient" graph map $f:\G\to\G$, called a \emph{train track map} \cite{bh92}. Here $\G$ is \emph{marked} by an identification $\pi_1(\G)\cong F_r$. The stretch factor $\lambda(\vphi)$ is then equal to the Perron-Frobenius eigenvalue $\lambda(f)$ of the transition matrix of $f$. We introduce a new $\out$-conjugacy class invariant $\mathbf U(\vphi)$ which counts the number of distinct combinatorial types of unmarked train track representatives $f\colon\G\to\G$ of a fully irreducible $\vphi\in \out$ such that the underlying graph $\G$ is an $r$-rose $R_r$, i.e. graph with a single vertex and Betti number $b_1(\G)=r$.  A priori, the upper bound for the cardinality of $\mathbf U(\vphi)$ is double exponential in $\log\lambda(\vphi)$, see Lemma~\ref{lem:upper-bound} below. However, there is an important class of fully irreducibles for which this bound is much better. These are the \cite{LoneAxes16} \emph{lone axis} fully irreducibles, which are so called ``ageometric" fully irreducibles $\vphi\in \out$ with a unique invariant axis $A_\vphi$ in $CV_r$. In \cite{LoneAxes16}, Mosher and Pfaff provided an ``ideal Whitehead graph" $\mathcal{IW}(\vphi)$ criterion for an ageometric fully irreducible $\vphi$ to have a lone axis. Here we show that, if $\vphi$ is a lone axis fully irreducible with a train track representative $f\colon R_r\to R_r $, then $\#\mathbf U(\vphi)\le ||f||$, where $||f||$ is the sum of the lengths of the edge-paths $f(e)$ as $e$ varies over the edges of $R_r$.

Now let $r\ge 3$ and let $X=\{x_1,\dots, x_r\}$ be a free basis of $F_r$. For a ``random" positive word $w(x_2,\dots,x_r)$ of length $e^L$ we construct an explicit positive automorphism $\psi_w$ of $F_r$ such that, when viewed as a train track map $g_w$ on the rose $R_r$, it satisfies $||g_w||\approx e^L$. We then precompose $g_w$ with another positive train track map $\eta: R_r\to R_r$ to get a train track map $f_w=g_w\circ \eta: R_r\to R_r$ representing an outer automorphism $\vphi_w\in\out$. Denote the set of all such $\vphi_w$ by $\mS_r$. The fact that $\eta$ does not depend on $w$ and $L$ implies that $||f_w||\le C e^L$. The number of distinct ``random'' positive words  $w(x_2,\dots,x_r)$ of length $e^L$ is on the order of $(r-1)^{e^L}$, which gives us on the order of  $(r-1)^{e^L}$ combinatorially distinct unmarked train track maps $f_w: R_r\to R_r$. The key step is to choose $\eta$ in such a way that for all $w$ as above, $\vphi_w$ is a lone axis ageometric fully irreducible. Technically, this is the hardest part of the proof since satisfying the lone axis property for $\vphi_w$ requires, among other things, that $f_w$ have no periodic Nielsen paths (PNPs). Here we rely on train track automata ($\mathcal{ID}$ diagrams) and PNP prevention technology developed by Pfaff in~\cite{Thesis,IWGII}. Once we know that each $\vphi_w$ is lone axis, the above estimate for the size of $\mathbf U$ implies that $\#\mathbf{U}(\vphi_w)\le Ce^L$. Thus the maps $f_w$ give approximately $(r-1)^{e^L}$ combinatorially distinct train track maps on roses representing ageometric fully irreducibles $\vphi_w$, each with  $\#\mathbf{U}(\vphi_w)\le Ce^L$. Therefore,
\[
\#\mS_r\ge const \frac{(r-1)^{e^L}}{Ce^L}\ge_{L\to\infty} (r-1.5)^{e^L} \tag{$\ddag\ddag$}.
\]
The fact that each $\vphi_w$ has a train track representative $f_w$ with $||f_w||\le C e^L$ implies that $\log \lambda(\vphi_w)\le L+\log C$, which, together with ($\ddag$), leads to the lower bound in Theorem~\ref{thm:main}. The upper bound in Theorem~\ref{thm:main} is much easier, and is obtained by a Perron-Frobenius counting argument estimating from above the number of train track maps $f$ with  $\log\lambda(f)\le L$; see Lemma~\ref{lem:upper-bound} below.

Ultimately, the proof of Theorem~\ref{thm:main} relies on the fact that for $r\ge 3$ the number of conjugacy classes of primitive elements in $F_r$ of length $\le L$ grows exponentially in $L$. By contrast, the number of simple closed geodesics of length $\le L$ on a closed hyperbolic surface grows polynomially in $L$~\cite{Mi08}.

\subsection{Related results}

The first and only, before this paper, result about counting conjugacy classes of fully irreducibles was obtained in a recent paper of Hull and Kapovich~\cite{hk17}. They proved, roughly, that for $r\ge 3$ the number of distinct $\out$-conjugacy classes $[\vphi]$ of fully irreducibles $\vphi$ coming from a ball of radius $L$ in the Cayley graph of $\out$, and with $\log\lambda(\vphi)$ on the order of $L$, grows exponentially in $L$. The paper also provides an informal heuristic argument for why one might expect that the total number of $\out$-conjugacy classes $[\vphi]$ of fully irreducibles $\vphi$ with $\log\lambda(\vphi)\le L$ grows doubly exponentially in $L$. Here we prove that this is indeed the case.

Also, we have been informed by Kasra Rafi that during the 2016 MSRI special semester, Bestvina and Rafi came up with an argument (unpublished), for proving that the number of distinct fully irreducibles $\vphi\in \out$ with $\log\lambda(\vphi)\le L$ grows doubly exponentially in $L$.  Their argument was based on a roughly similar construction to that outlined in the proof sketch above (although not using lone axis fully irreducibles).
However, unlike our Theorem~\ref{thm:main}, their argument only concerned counting fully irreducible elements in $\out$ with $\log\lambda\le L$ rather than counting $\out$-conjugacy classes of such fully irreducibles.

A new paper of Gekhtman and Yang~\cite{GY22} obtains a general result about counting conjugacy classes in groups acting on proper metric spaces with contracting elements and with finite critical exponent. Their result does not directly apply to our situation since the natural metric on $CV_r$ is asymmetric and the action of $\out$ on $CV_r$ has infinite critical exponent.  Still, it may be possible to extend and generalize the methods of \cite{GY22} to obtain an alternate proof of our Theorem~\ref{thm:main}.

\vskip20pt

\section{Definitions and Background}{\label{s:Dfs}}

\vskip10pt

We assume throughout this paper that $r\ge 2$ is an integer, and that $F_r$ is the rank-$r$ free group with a fixed free basis $X=\{x_1, \dots, x_r\}$. $R_r$ will denote the $r$-petaled rose, i.e. the graph with $r$ loop-edges at a single vertex. We fix an orientation on $R_r$ and an identification of each positive edge $e_i$ of $R_r$ with an element $x_i$ of $X$.

\vskip10pt

\subsection{Train track maps \& (fully) irreducible outer automorphisms.}

\vskip1pt

This paper follows the conventions and formalism regarding graphs and graph maps explained in detail in \cite{dkl15}. In particular, unless specified otherwise, graphs are 1-dimensional CW-complexes equipped with a ``linear atlas of charts" on edges and all graph maps, topological representatives, and train track maps are assumed to be ``linear graph maps", in the terminology of \cite{dkl15}. Basically these assumptions translate to working in the PL category, ruling out various pathologies for fixed and periodic points of train track maps in relation to Nielsen paths and periodic Nielsen paths. We refer the reader to \cite{dkl15} for more details (not important for us here).


For a graph $\G$, we denote by $E\G$ the (oriented) edge set of $\G$ and denote by $V\G$ the vertex set of $\G$. For an edge-path $\gamma$ in $\G$ we use $|\gamma|$ to denote the combinatorial length of $\gamma$, that is the number of edges $\gamma$ traverses (with multiplicity).

\begin{df}[Graph maps \& train track maps]\label{d:Maps}
We call a continuous map of graphs $g \colon \Gamma \to \Gamma'$ a \emph{graph map} if $g$ takes vertices to vertices and edges to edge-paths of positive length.

For a finite connected graph $\Gamma$ with no degree 1 vertices, a graph map $g \from \Gamma \to \Gamma$ is a \emph{train track map} if $g$ is a homotopy equivalence and if for each $k\ge 1$ the map $g^k$ is locally injective on edge interiors. In particular, all of our train track maps $g:\G\to\G$ are necessarily surjective. We call a train track map $g:\G\to\G$ \emph{expanding} if for each edge $e \in E\G$ we have that $|g^n(e)|\to\infty$ as $n\to\infty$. We call $g$ \emph{irreducible} if for any edges $e,e'\in E\G$ there is some $k\ge 1$ such that the path $f^k(e')$ contains $e$ or $\bar{e}$.

If $\vphi\in \out$, and $\G$ is a finite connected graph with no degree-1 vertices, equipped with a \emph{marking} (i.e. a homotopy equivalence $m \from R_r \to \Gamma$), and $g \from \Gamma \to \Gamma$ is a graph map such that $g_{*}=\vphi$, then we say $g$ \emph{represents} $\vphi$.

Note that, via our identification of $E(R_r)$ with the free basis $X$, from an automorphism $\Phi\in Aut(F_r)$ we obtain an \emph{induced} graph map $R_r\to R_r$ sending $e_i$ to the edge path $e_{i,1} \dots e_{i,k}$ where $\Phi(x_i)=x_{i,1} \dots x_{i,k}$. (This map is a representative of $\vphi$, the outer class of $\Phi$.) We may sometimes blur the distinction between an automorphism and the graph map it induces.
\end{df}

\begin{df}[Directions]\label{d:Directions} Let $x\in V(\Gamma)$. The \emph{directions} at $x$ are the germs of initial segments of edges emanating from $x$. For each directed edge $e \in E\G$, we let $D(e)$, or just $e$, denote the initial direction of $e$. For an edge-path $\gamma=e_1 \dots e_k$, define $D \gamma := D(e_1)$. Let $g \from \Gamma \to \Gamma$ be a graph map. Then denote by \emph{$Dg$} the map of directions induced by $g$, i.e. $Dg(d)=D(g(e))$ for $d=D(e)$. A direction $d$ is \emph{periodic} if $Dg^k(d)=d$ for some $k>0$ and \emph{fixed} when $k=1$.
\end{df}

\begin{df}[Turns]\label{d:GateStructures} An unordered pair of directions $\{d_i, d_j\}$ at a common vertex is called a \emph{turn}, and a \emph{degenerate turn} if $d_i = d_j$. If $\gamma_1$ and $\gamma_2$ are paths in $\Gamma$ initiating at a common vertex, then by $\{\gamma_1, \gamma_2\}$ we will mean $\{D\gamma_1, D\gamma_2\}$. For a path $\gamma=e_1e_2 \dots e_{k-1}e_k$ in $\Gamma$ where $e_1$ and $e_k$ may be partial edges, we say $\gamma$ \emph{takes} $\{\overline{e_i}, e_{i+1}\}$ for each $1 \leq i < k$. An edge-path which takes no degenerate turns is called \emph{tight}, as is a homotopy equivalence whose edge images are each tight. For both edges and paths we more generally use an ``overline'' to denote a reversal of orientation.

Let $g \colon \Gamma \to \Gamma$ be a graph map. Denote also by $Dg$ the map induced by $Dg$ on the turns of $\Gamma$.
A turn $\tau$ is called g-\emph{prenull} if $Dg(\tau)$ is degenerate.
The turn $\tau$ is called an \emph{illegal turn} for $g$ if $Dg^k(\tau)$ is degenerate for some $k$ and a \emph{legal turn} otherwise. A path $\gamma$ is \emph{$g$-legal} if each turn of $\gamma$ is $g$-legal.

A turn $T$ in $\Gamma$ is \emph{$g$-taken} if there exists an edge $e$ so that $g(e)$ takes $T$. We use $\mT(g)$ to denote the set of $g$-taken turns and define $\mT_{\infty}:=\cup_{k\ge 1}\mT(g^k)$.
\end{df}

\begin{df}[Transition matrix $M(g)$, Perron-Frobenius matrix, Perron-Frobenius eigenvalue] The \emph{transition matrix} $M(g)$ of a train track map $g \from \Gamma \to \Gamma$ is an $m\times m$ integer matrix (where $m$ is the number of positive edges in $E\G$) $M(g)=(m_{ij})$ such that for each $1\le i,j\le m$, the entry $m_{ij}$ the number of times the path $g(e_i)$ traverses $e_j$ in either direction.

A transition matrix $M=(m_{ij})$ is \emph{Perron-Frobenius (PF)} if there exists a $k\ge 1$ such that $M^k$ is strictly positive.
By Perron-Frobenius theory, we know that each such matrix has a unique eigenvalue of maximal modulus and that this eigenvalue is real and $> 1$. 
This eigenvalue is called the \emph{Perron-Frobenius (PF) eigenvalue} of $M$, and for $M(g)$ is denoted $\lambda(g)$.
\end{df}

\begin{df}[Stretch factor of a fully irreducible]
Given a free basis $X$ of $F_r$ and element $w\in F_r$, denote by $||w||_X$ the cyclically reduced length of $w$ with respect to $X$.
For a fully irreducible $\vphi\in\out$, free basis $X$, and $1\ne w\in F_r$ such that the conjugacy class $[w]$ is not $\vphi$-periodic, it is known~\cite{bogop08} that the limit
\[
\lim_{n\to\infty} \sqrt[n]{||\vphi^n(w)||_X}
\]
exists and is independent of $X$ and $w$.  This limit is called the \emph{stretch factor} of $\vphi$ and is denoted $\lambda(\vphi)$.

It is known that if $\vphi\in\out$ is fully irreducible then $\vphi$ admits a train track representative and that every such train track representative is expanding and has a Perron-Frobenius transition matrix. Moreover, if $g:\G\to\G$ is a train track representative of $\vphi$ then $\lambda(g)=\lambda(\vphi)$ (see, for example,~\cite{bogop08}). 
\end{df}

\vskip10pt

\subsection{Periodic Nielsen paths}{\label{ss:PNPs}}

\vskip1pt

\begin{df}[Nielsen paths \& rotationless powers] Let $g \colon \Gamma \to \Gamma$ be an expanding irreducible train track map. Bestvina and Handel \cite{bh92} define a nontrivial immersed path $\rho$ in $\Gamma$ to be a \emph{periodic Nielsen path (PNP)} if $g^R(\rho) \cong \rho$ rel endpoints for some power $R \geq 1$ (and just a \emph{Nielsen path (NP)} if $R=1$). A NP $\rho$ is called \emph{indivisible} (hence is an ``iNP'') if it cannot be written as $\rho = \gamma_1\gamma_2$, where $\gamma_1$ and $\gamma_2$ are themselves NPs.

By \cite[Corollary 4.43]{fh11}, for each $r \geq 2$, there exists an $R(r) \in \NN$ such that for each expanding irreducible train track representative $g$ of each outer automorphism $\vphi \in \out$, each PNP for $g$ is an NP for $g^{R(r)}$. This power $R$ is called the \emph{rotationless} power.
\end{df}

We remark that iNPs have a specific structure, described in \cite[Lemma 3.4]{bh92}:

\begin{prop}{\label{P:PINP}}
Let $g \colon \Gamma\to\Gamma$ be an expanding irreducible train track map. Then every iNP $\rho$ in $\Gamma$ has the form $\rho= \overline{\rho_1}\rho_2$, where $\rho_1$ and $\rho_2$ are nondegenerate legal paths sharing their initial vertex $v\in \Gamma$ and such that the turn at $v$ between $\rho_1$ and $\rho_2$ is an illegal nondegenerate turn for $g$.
\end{prop}

\begin{df}[Ageometric]{\label{d:ageometric}} A fully irreducible outer automorphism is called \emph{ageometric} if it has a train track representative with no PNPs.
\end{df}

The following proposition is a collection of arguments implicitly and explicitly given in a collection of other papers which we collect here in order to justify our procedure for finding (or ruling out the existence of) periodic Nielsen paths.

\begin{prop}{\label{P:AllLemmas}} Suppose that $g_1, \dots, g_n$ are tight homotopy equivalences satisfying:\\
\vspace{-5mm}
\begin{enumerate}
\item[1.] Each $g_k$ has precisely one prenull turn $\tau_k$
\item[2.] and with indices viewed mod n, some direction in $\tau_k$ is not in the image of $Dg_{k-1}$. 
\end{enumerate}
Suppose further that $g=g_n\circ\cdots\circ g_1$ is a train track map.
For each $1\leq k\leq n$, let $f_k:=g_{k-1}\circ\cdots\circ g_1\circ g_n\circ\cdots\circ g_k$.

Then\\
\vspace{-6mm}
{\begin{itemize}
  \item[a.] For each $1\leq k\leq n$, the prenull turn for $g_k$ is the only illegal turn for $f_k$.
  \item[b.] If $\rho$ were an iNP for $g$, then for each $k$:
      {\begin{itemize}
  \item[i.] $((g_{k-1}\circ\cdots\circ g_1)(\rho))_{\#}$ is an iNP for $f_k$.
  \item[ii.] $((g_{k-1}\circ\cdots\circ g_1)(\rho))_{\#}$ contains the prenull turn for $g_k$.
  \item[iii.] It is possible to write $\rho=\overline{\gamma_1}\gamma_2$ where, for each $j\in\{1,2\}$, we have $\gamma_j=\alpha_j\sigma_j$ with\\ 
      $\bullet$ $\sigma_j$ is $g$-legal,\\ 
      $\bullet$  $(g_{k-1}\circ\cdots\circ g_1)(\alpha_1)=(g_{k-1}\circ\cdots\circ g_1)(\alpha_2)$,\\ 
      $\bullet$  and $\{D(g_{k-1}\circ\cdots\circ g_1)(\sigma_1), D(g_{k-1}\circ\cdots\circ g_1)(\sigma_2)\}$ is the $g_k$-prenull turn $\tau_k$.
\end{itemize}}
\end{itemize}}
\end{prop}

\vskip10pt

\begin{proof} Throughout this proof we use $\simeq$ to denote homotopy equivalence relative to endpoints.

(a)  Suppose, for the sake of contradiction, that some $f_k$ has more than one illegal turn. Since the direction map $Df_k^L$ for $f_k^L$ is a composition of the $Dg_j$, this would imply that, with indices viewed mod $n$, some $g_j$ would have a prenull turn in which both directions are in the image of $Dg_{j-1}$. However, by assumption (2), this does not occur for any $g_j$.

(b-i) Let $\rho_k$ denote $((g_{k-1}\circ\cdots\circ g_1)(\rho))_{\#}$. Then by the definitions and the fact that $\rho$ is an iNP for $g$, was have 
$$f_k(\rho_k)\simeq f_k((g_{k-1}\circ\cdots\circ g_1)(\rho))\simeq (f_k\circ g_{k-1}\circ\cdots\circ g_1)(\rho)\simeq (g_{k-1}\circ\cdots\circ g_1\circ g)(\rho))\simeq$$
$$(g_{k-1}\circ\cdots\circ g_1)(g(\rho))\simeq (g_{k-1}\circ\cdots\circ g_1)(\rho)\simeq \rho_k.$$
\noindent So $\rho_k$ is an NP for $f_k$. 

By symmetry in the argument, we in fact have that if $\rho_k'$ is an iNP for $f_k$, then $((g_n\circ\cdots\circ g_k)(\rho_k'))_{\#}$ is a NP for $g$. Thus, a decomposition of $\rho_k$ into $>1$ iNP would yield a decomposition of $\rho$ into more than 1 NP (it is not difficult to see they are unique), contradicting that $\rho$ is an iNP.

(b-ii) If $\rho$ were an iNP for $g$, then by (b-i) we would have that $((g_{k-1}\circ\cdots\circ g_1)(\rho))_{\#}$ would be an iNP and hence would have an $f_k$ illegal turn (by Proposition \ref{P:PINP}). By (a) this must be the prenull turn for $g_k$.

(b-iii) Since $\rho$ is an iNP, by Proposition \ref{P:PINP}, there exist $g$-legal paths $\gamma_1$ and $\gamma_2$ so that $\overline{\gamma_1}\gamma_2$. We first show that for each $j\in\{1,2\}$ the path $(g_{k-1}\circ\cdots\circ g_1)(\gamma_j)$ is $f_k$-legal. Suppose for the sake of contradiction that for some $j\in\{1,2\}$ and $p>0$ we had that $f_k^p((g_{k-1}\circ\cdots\circ g_1)(\gamma_j))$ had cancellation. Then we would have that $g^{p+1}(\gamma_j)=(g_{n}\circ\cdots\circ g_k)(f_k^p((g_{k-1}\circ\cdots\circ g_1)(\gamma_j)))$ would necessarily also, contradicting the $g$-legality of $\gamma_j$.

Since subpaths of legal paths are legal, the result of (b-iii) then follows from (b-ii).
\end{proof}

\vskip10pt

\subsection{Whitehead graphs}{\label{ss:wg}}

\vskip1pt

Throughout this subsection, $g \from R_r \to R_r$ will be a PNP-free expanding irreducible train track map and $v$ will be the vertex of $R_r$. More general definitions can be found in \cite{hm11} or \cite{LoneAxes16}, with explanations in \cite{Thesis} of the reduction to our setting.

\vskip10pt

\begin{df}[Whitehead graphs \& index]{\label{d:WhiteheadGraphs}} The \emph{local Whitehead graph}
$LW(g)$ has a vertex for each direction at $v$ and an edge connecting the vertices corresponding to a pair of directions $\{d_1,d_2\}$ when $\{d_1,d_2\}\in \mT_{\infty}$. The \emph{stable Whitehead graph} $SW(g)$ is the subgraph of $LW(g)$ obtained by restricting to the vertices of $LW(g)$ corresponding to periodic directions.

If $g$ further represents a fully irreducible $\vphi \in \out$, then the \emph{ideal Whitehead graph $\mathcal{IW}(\vphi)$ of $\vphi$} is isomorphic to $SW(g)$. Justification of this being an outer automorphism invariant can be found in \cite{hm11, Thesis}. From the ideal Whitehead graph, one can obtain the \emph{rotationless index} $i(\vphi) :=  1-\frac{k}{2}$, where $k$ is the number of vertices of $\mathcal{IW}(\vphi)$.
\end{df}

\vskip10pt

The following lemma follows directly from the definitions:

\begin{lem}\label{lem:comp}
Let $\Gamma$ be a finite connected graph and let $g,h:\Gamma\to\Gamma$ be graph maps such that $g$ is surjective. Then
\[
\mT(h\circ g)=\mT(h)\cup Dh(\mT(g)).
\]
\end{lem}

We also need the following lemma, whose proof, while not explicitly given in either \cite[Lemma 3.7]{kp15} or \cite{Thesis}, is similar to the proofs in each. We include a proof for completeness.

\begin{lem}{\label{l:CombinedTurns}}
Let $n\ge 1$ and let $h_1, \cdots, h_n$ be train track maps on the rose $R_r$ (in the sense of Definition \ref{d:Maps} so that, in particular they are homotopy equivalences, and thus surjective). 


Then
$$\mT(h_n\circ\cdots\circ h_1)=[\mT(h_n)] \bigcup_{k=1}^{n-1} [D(h_n\circ\cdots\circ h_{k+1})(\mT(h_k))].$$
\end{lem}

\begin{proof}
We argue by induction on $n$. For $n=1$ the statement vacuously holds.

Suppose now that $n>1$ and the statement is known to hold for $n-1$. Thus 
\[
\mT(h_{n-1}\circ\cdots\circ h_1)=[\mT(h_{n-1})] \bigcup_{k=1}^{n-2} [D(h_{n-1}\circ\cdots\circ h_{k+1})(\mT(h_k))].
\]

Note that by assumption each $h_i:R_r\to R_r$ is surjective and hence so is $h_{n-1}\circ\cdots\circ h_1$.
Hence, by Lemma~\ref{lem:comp}, 
\[
\mT(h_n\circ\cdots\circ h_1)=\mT(h_n\circ (h_{n-1}\cdots\circ h_1))=[\mT(h_n)]  \cup Dh_n (\mT(h_{n-1}\circ\cdots\circ h_1)).
\]
  
Using the inductive hypothesis, we have
\begin{gather*}
Dh_n(\mT(h_{n-1}\circ\cdots\circ h_1))=Dh_n\left( [\mT(h_{n-1})] \bigcup_{k=1}^{n-2} [D(h_{n-1}\circ\cdots\circ h_{k+1})(\mT(h_k))]  \right)=\\
Dh_n(\mT(h_{n-1})) \bigcup_{k=1}^{n-2} Dh_n\left([D(h_{n-1}\circ\cdots\circ h_{k+1})(\mT(h_k))]\right)=\\
Dh_n(\mT(h_{n-1})) \bigcup_{k=1}^{n-2} \left([Dh_n\circ D(h_{n-1}\circ\cdots\circ h_{k+1})(\mT(h_k))]\right)=\\
\bigcup_{k=1}^{n-1} [D(h_n\circ\cdots\circ h_{k+1})(\mT(h_k))].
\end{gather*}
 
Hence
\[
\mT(h_n\circ\cdots\circ h_1)=[\mT(h_n)] \bigcup_{k=1}^{n-1} [D(h_n\circ\cdots\circ h_{k+1})(\mT(h_k))],
\]
as required. 
This completes the inductive step and the proof of the lemma.
\end{proof}

\vskip10pt

\subsection{Full irreducibility criterion.}

\vskip1pt

Proposition \ref{prop:FIC} is stated as such in \cite{stablestrata}. It is essentially \cite[Proposition 4.1]{IWGII}, with the added observation that a fully irreducible outer automorphism with a PNP-free train track representative is in fact ageometric (by definition). Kapovich~\cite{k14} has a related result.

\begin{prop}[\cite{stablestrata}](\emph{A Special Case of the Ageometric Full Irreducibility Criterion (FIC)})\label{prop:FIC}
Let $g\colon R_r \to R_r$ be a PNP-free, irreducible train track representative of $\vphi \in \out$. Suppose that $M(g)$ is Perron-Frobenius and that the local Whitehead graph $v$ at the vertex $v$ of $R_r$ is connected. Then $\vphi$ is an ageometric fully irreducible outer automorphism.
\end{prop}

\vskip10pt

\section{Fold lines in Outer space}{\label{s:Geodesics}}

\vskip10pt

For each integer $r\geq 3$, we use $\os$ to denote the rank-$r$ Culler-Vogtmann Outer space (as defined in \cite{cv86}). One may consult \cite{ParkCity14} for a nice survey on the topic.

\vskip10pt

\subsection{Stallings fold decompositions}{\label{ss:StallingsFoldDecompositions}}

\vskip1pt

\begin{df}[Stallings fold decomposition]{\label{d:StallingsFoldDecomposition}} \cite{s83}. Let $g \colon \Gamma \to \Gamma'$ be a surjective homotopy equivalence graph map. Let $e_1' \subset e_1$ and $e_2' \subset e_2$ be maximal, initial, nontrivial subsegments of edges $e_1$ and $e_2$ of $\Gamma$ emanating from a common vertex and satisfying that $g(e_1')=g(e_2')$ as edge paths and that the terminal endpoints of $e_1'$ and $e_2'$ are distinct points in $g^{-1}(V\Gamma')$. Redefining $\Gamma$ to have vertices at the endpoints of $e_1'$ and $e_2'$ if necessary, one can obtain a graph $\Gamma_1$ by identifying the points of $e_1'$ and $e_2'$ that have the same image under $g$, a process called \emph{folding}.
Stallings \cite{s83} showed that such a homotopy equivalence graph map $g$ factors as a composition of folds and a final homeomorphism, i.e. has a \emph{Stallings fold decomposition}. One way to view the process, important for the proof of Lemma \ref{lem:multipl}, is to ``label'' each edge $e$ in $\G$ by its image path $g(e)$ and to iteratively perform folds of identically labeled edge segments emanating from a common vertex.
\end{df}

In \cite{s89}, Skora interpreted a Stallings fold decomposition for a graph map homotopy equivalence $g\colon \Gamma \to \Gamma'$ as a sequence of folds performed continuously. For a train track map $g$, by repeating this procedure, one obtains a geodesic in $\os$ called a \emph{periodic fold line} for $g$. It is proved in \cite[Lemma 2.27]{stablestrata} that, if $g$ is a train track map, then the periodic fold line is in fact a geodesic.

\begin{rk} In general, each fully irreducible $\vphi\in\os$ has many train track  representatives, each of which can have several distinct Stallings fold decompositions (and hence several associated periodic fold lines).
\end{rk}

\vskip10pt

\subsection{Lone axis fully irreducible outer automorphism}{\label{ss:LoneAxisFullyIrreducibles}}

\vskip1pt

In \cite{hm11}, Handel and Mosher define the axis bundle for an ageometric fully irreducible outer automorphism $\vphi$. \cite{hm11} contains three equivalent definitions. Definition \ref{d:ab} below provides a fourth equivalent definition (proving the equivalence of this definition to the others is straight-forward).

\begin{df}[Axis bundle $\mA_{\vphi}$]{\label{d:ab}}
Let $\vphi\in \out$ be an ageometric fully irreducible outer automorphism. Then $\mA_{\vphi}$ is the closure of the union of the images of the periodic fold lines for train track representatives of positive powers of $\vphi$.
\end{df}

We say that an ageometric fully irreducible $\vphi \in \out$ is a \emph{lone axis fully irreducible} if the axis bundle $\mA_{\vphi}$ consists of a single fold line (which is then necessarily $\vphi$-periodic). Note that if $\vphi \in \out$ is a lone axis fully irreducible, then so is every positive power of $\vphi$.

The following theorem is proved in \cite{LoneAxes16}. The statement given here is in fact a combination of \cite[Theorem 4.6]{LoneAxes16}, \cite[Theorem 4.7]{LoneAxes16}, and \cite[Lemma 4.5]{LoneAxes16}.

\begin{thm}[\cite{LoneAxes16}]{\label{t:uniqueaxis}} Let  $\vphi \in \out$ be an ageometric fully irreducible outer automorphism. Then $\vphi$ is a lone axis fully irreducible outer automorphism if and only if ~\\
\vspace{-5mm}
\begin{enumerate}
\item the rotationless index satisfies $i(\vphi) = \frac{3}{2}-r$ and 
\item no component of the ideal Whitehead graph $\mathcal{IW}(\vphi)$ has a cut vertex.
\end{enumerate}
In this case the axis bundle $\mA_{\vphi}$ is a unique $\vphi$-periodic fold line; in particular, $\mA_{\vphi}$ is the unique periodic fold line constructed from any train track representative of each positive power of $\vphi$.
\end{thm}

\vskip20pt

\section{The Automorphisms}\label{s:autos}

\vskip10pt

We continue assuming that $r\ge 3$, that $F_r=F(x_1, \dots, x_r)$, and that $R_r$ is the $r$-petaled rose. We may blur the distinction between an automorphism and its induced map on the rose.

\vskip10pt

\begin{df}[Full word]
Recall that a positive word is one that can be written as a product of generators $x_1,\dots, x_r$ without use of any inverses $x_j^{-1}$. We say that a positive word $w(x_2,\dots, x_r)$ is \emph{full} if for all $2\le i,j\le r$ the word $x_ix_j$ occurs as a subword of $w$. That is, a positive word $w(x_2,\dots, x_r)$ is full if and only if it contains each turn $\{d_1,d_2\}$ with
$d_1 \in\{x_2, \dots, x_r\}$ and
$d_2 \in\{\ol{x_2}, \dots, \ol{x_r}\}$.
\end{df}

Note that for every full word $w(x_2,\dots,x_r)$ we have $|w|\ge 2(r-1)\ge r\ge 2$.

\begin{df}[$g_w$]\label{d:gw} Let $r\ge 3$ and let $w(x_2, \dots, x_r)$ be a full positive word in $x_2, \dots, x_r$ starting with $x_{r-1}$ and ending with $x_2$. We then define graph maps $g_w:R_r\to R_r$ by:

~\\
\vspace{-3mm}
$$
g_w(x_k) =
\begin{cases}
x_{k+1}\quad if \quad 1\le k \le r-1 \\
 x_1 w(x_2, \dots, x_r)\quad if \quad k=r
\end{cases}
$$
\end{df}

\begin{df}[$g_k$, $g_{k,i}$]\label{d:gk} For each $\veps\in\{0,6\}$, we define the following automorphisms (all generators whose images are not explicitly given are fixed):
\[ g_{1+\veps}: x_1 \mapsto x_1 x_r, \quad
g_{2+\veps}: x_r \mapsto x_r x_1, \quad
g_{3+\veps}: x_r \mapsto x_r x_{r-1}
\]
\[
g_{4+\veps}: x_{r-1}\mapsto x_{r-1} x_r, \quad
g_{5+\veps}: x_{r-1}\mapsto x_{r-1} x_1, \quad
g_{6+\veps}: x_1\mapsto x_1 x_{r-1}.
\]

Denoting the identity map on $R_r$ by $id_{R_r}$, for each $1 \le i, k\le 12$, let
$$
g_{k,i} =
\begin{cases}
g_k \circ \cdots \circ g_i \text{ if } k>i \\
g_k  \text{ if } i=k \\
id_{R_r} \text{ if } i>k
\end{cases}
$$
\end{df}

\begin{rk} The sequence of $g_i$ was constructed as a loop in an AM diagram. The interested reader may find the definition of an AM diagram (called an $\mathcal{ID}$ diagram in later papers) in \cite{Thesis} Chapter 9. We used $(g_6\circ \dots \circ g_1)^2$ to fix periodic directions, making computations and arguments simpler.
\end{rk}

\vskip10pt

\begin{df}[$g$]\label{df:g}
Choose a full positive word $w(x_2, \dots, x_r)$ starting with $x_{r-1}$ and ending with $x_2$ and put 
\[
g:=g_w\circ g_{12,1}=g_w\circ g_{12}\circ \dots \circ g_1:R_r\to R_r.
\]

\end{df}

\begin{rk}[Compositions are train track maps]\label{r:tt} Notice that, since each $g_w$ and each $g_i$ represent positive automorphisms of $F_r$, any composition of them is also a positive automorphism. Hence, any such composition induces a train track map on the rose $R_r$ representing an outer automorphism $\phi_w\in \out$. Moreover, the definition of $g$ implies that for $i=1,\dots, r$ we have $|g^r(x_i)|\ge r$ and the word $g^r(x_i)$ involves each of the generators $x_1,\dots, x_r$ of $F_r$. Thus $g:R_r\to R_r$ is an expanding irreducible train track map.
\end{rk}

We record here a table of data about the $g_k$ used in proofs that follow.

\begin{table}[h]
\begin{tabular}{|l|l|l|l|l|l|l|}
\hline 
 &  &  &  &  & &  \\
 & \quad $g_1$ & \quad $g_2$ & \quad $g_3$ & \quad $g_4$ & \quad $g_5$ & \quad $g_6$  \\
 &  &  &  &  & &  \\
 & $x_1 \mapsto x_1 x_r$ & $x_r \mapsto x_r x_1$  & $x_r \mapsto x_r x_{r-1}$ & $x_{r-1}\mapsto x_{r-1} x_r$ & $x_{r-1}\mapsto x_{r-1} x_1$ & $x_1\mapsto x_1 x_{r-1}$ \\
 &  &  &  &  & &  \\
 \hline 
 &  &  &  &  & &  \\
\text{Prenull turn} & $\{\ol{x_1},\ol{x_r}\}$ & $\{\ol{x_1},\ol{x_r}\}$ & $\{\ol{x_{r-1}},\ol{x_r}\}$ & $\{\ol{x_{r-1}},\ol{x_r}\}$ & $\{\ol{x_1},\ol{x_{r-1}}\}$ & $\{\ol{x_1},\ol{x_{r-1}}\}$ \\
 &  &  &  &  & &  \\
\hline 
 &  &  &  &  & &  \\
\text{New turn} & $\{\ol{x_1},{x_r}\}$ & $\{\ol{x_r},{x_1}\}$  & $\{\ol{x_r}, x_{r-1}\}$ & $\{\ol{x_{r-1}}, x_r\}$ & $\{\ol{x_{r-1}}, x_1 \}$ & $\{\ol{x_1}, x_{r-1} \}$ \\
\text{created} &  &  &  &  & &  \\
 &  &  &  &  & &  \\
\hline 
 &  &  &  &  & &  \\
\text{Direction map} & $\ol{x_1}\mapsto \ol{x_r}$  & $\ol{x_r}\mapsto \ol{x_1}$  & $\ol{x_r}\mapsto \ol{x_{r-1}}$ & $\ol{x_{r-1}}\mapsto \ol{x_r}$ & $\ol{x_{r-1}}\mapsto \ol{x_1}$ & $\ol{x_1}\mapsto \ol{x_{r-1}}$ \\
 &  &  &  &  & &  \\
\hline 
 &  &  &  &  & &  \\
\text{Directions} & \quad $\ol{x_1}$ & \quad $\ol{x_r}$ & \quad $\ol{x_r}$ & \quad $\ol{x_{r-1}}$ & \quad $\ol{x_{r-1}}$ & \quad $\ol{x_1}$ \\
\text{not in image} &  &  &  &  & &  \\
 &  &  &  &  & &  \\
\hline 
\end{tabular}
\end{table}

\begin{lem}\label{l:WGs} For $g$ as above, $LW(g)$ consists of an edge connecting the pair $\{\ol{x_1},x_{r-1}\}$ together with the complete bipartite graph on the partition $\{\{x_1,\dots, x_r\},\{\ol{x_2},\dots, \ol{x_r}\}\}$.
\end{lem}

\vskip10pt

\begin{proof}
We begin with the following observations:
\begin{enumerate}
\item $$\mT_{\infty}(g)=\bigcup_{\ell\ge 1}[D(g^{\ell-1}\circ g_w)(\mT(g_{12,1}))\cup Dg^{\ell-1}(\mT(g_w))].$$
\item $Dg_w$ is defined by $x_k \mapsto x_{(k+1 \text{ mod r})}$ for $1\le k\leq r$, and $\ol{x_k} \mapsto \ol{x_{k+1}}$ for $1\le k\leq r-1$, and $\ol{x_r} \mapsto \ol{x_2}$.
\item $D(g_{12}\circ\cdots\circ g_1)$ is the identity map apart from $\ol{x_1}\mapsto\ol{x_r}$.
\item $Dg$ is defined by $x_k \mapsto x_{(k+1 \text{ mod r})}$ for $1\le k\leq r$, and $\ol{x_k} \mapsto \ol{x_{k+1}}$ for $1\le k\leq r-1$ and $\ol{x_r} \mapsto \ol{x_2}$.
\item $\mT(g_w)=\{ \{d_1,d_2\} \mid d_1\in \{x_2,\dots, x_r\} ~\&~ d_2\in \{\ol{x_2},\dots, \ol{x_r}\}\}\cup \{\{\ol{x_1},x_{r-1}\}\}$.
\item $\mT(g_{12,1})=\{ \{d_1,d_2\} \mid d_1\in \{x_1,x_{r-1}, x_r\} ~\&~ d_2\in \{\ol{x_{r-1}}, \ol{x_r}\}\}\cup \{\{\ol{x_1},x_{r-1}\}\}$ (see, for example, \cite{p15}).
\end{enumerate}

Since $Dg(\{\ol{x_{k-1}},x_r\})=\{\ol{x_k},x_1\}$ for all $2\le k \le r$ and $\{\ol{x_1},x_{r-1}\}\in\mT(g_w)$, we have
$$\{ \{x_1,d\} \mid d\in \{\ol{x_2},\dots, \ol{x_r}\}\}\cup \{\{\ol{x_1},x_{r-1}\}\}\subset \bigcup_{\ell\ge 1}[ Dg^{\ell-1}(\mT(g_w))]\subset \mT_{\infty}.$$

Together with $\mT(g_w)=\{ \{d_1,d_2\} \mid d_1\in \{x_2,\dots, x_r\} ~\&~ d_2\in \{\ol{x_2},\dots, \ol{x_r}\}\}\cup \{\{\ol{x_1},x_{r-1}\}\}\subset \mT_{\infty}$, this says
$$\{ \{d_1,d_2\} \mid d_1\in \{x_1,\dots, x_r\} ~\&~ d_2\in \{\ol{x_2},\dots, \ol{x_r}\}\}\cup \{\{\ol{x_1},x_{r-1}\}\}\subset \mT_{\infty}.$$

Since $\ol{x_1}$ is not in the image of $Dg$ or $Dg_w$, we then have
$$\mT_{\infty}=\{ \{d_1,d_2\} \mid d_1\in \{x_1,\dots, x_r\} ~\&~ d_2\in \{\ol{x_2},\dots, \ol{x_r}\}\}\cup \{\{\ol{x_1},x_{r-1}\}\}.$$
\qedhere
\end{proof}

\vskip15pt

In the proof of Lemma \ref{l:fi}, the procedure for showing that no iNPs exist is similar to that in \cite{IWGII}, \cite[Example 3.4]{hm11}, or \cite{p15}. For completion, we included proof of its validity in Proposition \ref{P:AllLemmas}. 

\vskip10pt

\begin{lem}\label{l:fi} Suppose that $w(x_2, \dots, x_r)$ is a full positive word starting with $x_{r-1}$ and ending with $x_2$.
Then $g=g_w\circ g_{12,1}$ represents an ageometric fully irreducible outer automorphism of $F_r$.
\end{lem}

\vskip10pt

\begin{proof}
By Remark \ref{r:tt}, $g$ is an expanding irreducible train track map. The transition matrix $M(g)$ is Perron-Frobenius since the image of each edge under $g^r$ contains every edge (including itself).

To show that $g$ represents an ageometric fully irreducible outer automorphism, we prove that $g$ additionally satisfies the remaining conditions of Proposition \ref{prop:FIC}, i.e $LW(g)$ is connected and $g$ is PNP-free. Since $LW(g)$ is connected by Lemma \ref{l:WGs}, we now show that $g$ has no PNPs.

Suppose for the sake of contradiction that $g$ has a PNP. Taking the rotationless power $g^R$, this gives an NP. Let $\rho = \overline{\rho_1} \rho_2$ be an iNP for $g^R$ in the decomposition of $\rho$ into iNPs, where $\rho_1 = e_1\dots e_m$ and $\rho_2 = e_1'\dots e_{m'}'$ are edge paths (with possibly $e_m$ and $e_{m'}'$ being partial edges) and with illegal turn $T = \{D(e_1), D(e_1')\}$ $=\{d_1, d_1'\}$. Notice first that each turn of $\rho_1$ and of $\rho_2$ must be in $\mT_{\infty}(g^R)$. And that $\{\ol{x_1},x_r\}\notin \mT_{\infty}(g)$. 

Considering $g_w$ as $g_{13}$ we are in the situation of Proposition \ref{P:AllLemmas} and rely heavily on results contained within it. Please also reference the table of \S \ref{s:autos}. In particular, by referencing the table, one can see that each $g_{1}, \dots, g_{12}$ has precisely one prenull turn $\tau_k$ and that one direction in $\tau_k$ is not in the image of $Dg_{k-1}$. By inspection of the direction map, one can see that the unique prenull turn for $g_{13}=g_w$ is $\tau_{13}=\{\ol{x_1},\ol{x_r}\}$ and that the direction $\ol{x_1}$ of $\tau_1$ is not in the image of $Dg_{13}=Dg_w$. Further, by the table one can see that direction $\ol{x_1}$ in $\tau_{13}$ is not in the image of $Dg_{12}=Dg_{6}$.

Since $\{\ol{x_1},\ol{x_r}\}$ is the only illegal turn for $g^R$, it would be necessary that (without loss of generality) $e_1=\ol{x_1}$ and $e_1'=\ol{x_r}$. Now:
$$g_1(\rho_1)=\bar{x_r} \bar{x_1}g_1(e_2)\dots g_1(e_m)
$$
$$
g_1(\rho_2)=\ol{x_r} g_1(e_2')\dots g_1(e_{m}').
$$
For $\rho$ to be an iNP, we need that $\{Dg_1(e_2'),\ol{x_1}\}$ is either degenerate or the prenull turn $\{\ol{x_1},\ol{x_r}\}$ for $g_2$. Since $\ol{x_1}$ is not in the image of $Dg_1$, this leaves that $Dg_1(e_2')=\ol{x_r}$. This would happen precisely when either $e_2'=\ol{x_1}$ or $e_2'=\ol{x_r}$. However, if $e_2'=\ol{x_1}$, then $\rho_2$ would contain turns not in $\mT_{\infty}(g)$. Hence, $e_2'=\ol{x_r}$.
 Now:
$$g_{2,1}(\rho_1)=\bar{x_1}\bar{x_r} \bar{x_1}g_{2,1}(e_2)\dots g_{2,1}(e_m)
$$
$$
g_{2,1}(\rho_2)=\bar{x_1}\bar{x_r}\bar{x_1}\bar{x_r} g_{2,1}(e_3')\dots g_{2,1}(e_{m}').
$$
So we need that $\{Dg_{2,1}(e_2),\ol{x_r}\}$ is either degenerate or the prenull turn $\{\ol{x_r},\ol{x_{r-1}}\}$ for $g_3$. Since $\ol{x_r}$ is not in the image of $Dg_{2,1}$, this leaves that $Dg_{2,1}(e_2)=\ol{x_{r-1}}$, i.e. $e_2=\ol{x_{r-1}}$.
 Now:
$$g_{3,1}(\rho_1)=\bar{x_1}\ol{x_{r-1}}\bar{x_r} \bar{x_1}  \ol{x_{r-1}}g_{3,1}(e_3)\dots g_{3,1}(e_m)
$$
$$
g_{3,1}(\rho_2)=
\bar{x_1}\ol{x_{r-1}}\bar{x_r}\bar{x_1}\ol{x_{r-1}}\bar{x_r}
g_{3,1}(e_3')\dots g_{3,1}(e_{m}').
$$
So we need that $\{Dg_{3,1}(e_3),\ol{x_r}\}$ is either degenerate or the prenull turn $\{\ol{x_r},\ol{x_{r-1}}\}$ for $g_4$. Since $\ol{x_r}$ is not in the image of $Dg_{3,1}$, this leaves that $Dg_{3,1}(e_3)=\ol{x_{r-1}}$, i.e. $e_3=\ol{x_{r-1}}$.
 Now:
$$g_{4,1}(\rho_1)=
\bar{x_1}\bar{x_r}\ol{x_{r-1}}\bar{x_r} \bar{x_1}\bar{x_r}\ol{x_{r-1}}
\bar{x_r}\ol{x_{r-1}}
g_{4,1}(e_4)\dots g_{4,1}(e_m)
$$
$$
g_{4,1}(\rho_2)=
\bar{x_1}\bar{x_r}\ol{x_{r-1}}\bar{x_r}\bar{x_1}\bar{x_r}\ol{x_{r-1}}\bar{x_r}
g_{4,1}(e_3')\dots g_{4,1}(e_{m}').
$$
So we need that $\{Dg_{4,1}(e_3'),\ol{x_{r-1}}\}$ is either degenerate or the prenull turn $\{\ol{x_1},\ol{x_{r-1}}\}$ for $g_5$. Since $\ol{x_{r-1}}$ is not in the image of $Dg_{4,1}$, this leaves that $Dg_{4,1}(e_3')=\ol{x_1}$, i.e. either $e_3'=\ol{x_r}$ or $e_3'=\ol{x_1}$. However, if $e_3'=\ol{x_1}$, then $\rho_2$ would contain turns not in $\mT_{\infty}(g)$. Hence, $e_3'=\ol{x_r}$.
 Now:
$$g_{5,1}(\rho_1)=
\bar{x_1}\bar{x_r}\bar{x_1}\ol{x_{r-1}}\bar{x_r} \bar{x_1}\bar{x_r}\bar{x_1}\ol{x_{r-1}}
\bar{x_r}\bar{x_1}\ol{x_{r-1}}
g_{5,1}(e_4)\dots g_{5,1}(e_m)
$$
$$
g_{5,1}(\rho_2)=
\bar{x_1}\bar{x_r}\bar{x_1}\ol{x_{r-1}}\bar{x_r}\bar{x_1}\bar{x_r}
\bar{x_1}\ol{x_{r-1}}\bar{x_r}
\bar{x_1}\bar{x_r}\bar{x_1}\ol{x_{r-1}}\bar{x_r}
g_{5,1}(e_4')\dots g_{5,1}(e_{m}').
$$
This tells us that the prenull turn for $g_6$ would have to be $\{\ol{x_{r-1}},\ol{x_r}\}$, but the prenull turn for $g_6$ is $\{\ol{x_1},\ol{x_{r-1}}\}$.
Thus, we have reached a contradiction and $g$ can have no PNPs.
\qedhere
\end{proof}

\vskip20pt

\begin{lem}\label{l:IWGs} Suppose that $w(x_2, \dots, x_r)$ is a full positive word in $x_2, \dots, x_r$ starting with $x_{r-1}$ and ending with $x_2$. Let $g:=g_w\circ g_{12,1}$ represent $\vphi_w\in \out$. Then $\mathcal{IW}(\vphi_w)$ is the complete bipartite graph on the partition $\{\{x_1,\dots, x_r\},\{\ol{x_2},\dots, \ol{x_r}\}\}$.
\end{lem}

\vskip10pt

\begin{proof}
Since there are no PNPs, we have $\mathcal{IW}(\vphi_w)\cong SW(g)$. Since all directions apart from $\ol{x_1}$ are periodic, $SW(g)$ is the graph obtain from $LW(g)$ by removing the vertex for $\ol{x_1}$. The result then follows from Lemma \ref{l:WGs}.
\qedhere
\end{proof}

\vskip10pt

\begin{prop}\label{prop:LoneAxes} Suppose that $w(x_2, \dots, x_r)$ is a full positive word, starting with $x_{r-1}$ and ending with $x_2$.
Then $g=g_w\circ g_{12,1}$ represents a lone axis ageometric fully irreducible $\vphi_w\in \out$.
\end{prop}

\vskip10pt

\begin{proof}
By Lemma \ref{l:fi}, $\vphi$ is an ageometric fully irreducible outer automorphism.
We can thus use Theorem \ref{t:uniqueaxis} to prove that $\vphi$ has a lone axis.

By Lemma \ref{l:IWGs}, the ideal Whitehead graph has a single component, which has $2r-1$ vertices. Thus, $i(\vphi_w)=\frac{3}{2}-r$, and so Theorem \ref{t:uniqueaxis}(1) is satisfied. Notice also that Lemma \ref{l:IWGs} implies that $\mathcal{IW}(\vphi_w)$ is a complete bipartite graph which has at least 2 vertices in each set of the partition. Hence, the only component of $\mathcal{IW}(\vphi_w)$ has no cut vertices. This implies that Theorem \ref{t:uniqueaxis}(2) is also satisfied. So $\vphi_w$ has a lone axis, as desired.
\qedhere
\end{proof}

\vskip20pt

\section{Counting}

For the remainder of this paper $\log x$ will denote the natural logarithm of $x$.

\begin{df}[$\mathbf U(\vphi)$]\label{df:U}
For $r\ge 3$ and $\vphi\in \out$ fully irreducible, $\mathbf U(\vphi)$ will denote the set of all unmarked train track representatives $f:R_r\to R_r$ of $\vphi$, considered as combinatorial graph maps.
More precisely, $\mathbf U(\vphi)$ is the set of equivalence classes of train track representatives of $\vphi$ based on an $r$-rose, equivalent when they differ by a change in marking and possibly a graph homeomorphism. That is, if $(f:\G\to \G, \alpha)$ and $(f':\G'\to \G', \alpha')$ are train track representatives of $\vphi$ on $r$-roses $\G$ and $\G'$, with markings $\alpha$ and $\alpha'$, these representatives are considered equivalent if there exists a homeomorphism $q:\G'\to \G$ such that $f'=q^{-1}\circ f\circ q$. (The existence of $q$ means that $f$ and $f'$ are the same as  combinatorial graph maps.)
\end{df}

\begin{rk}
Note that the set $\mathbf U(\vphi)$ is possibly empty (since not every fully irreducible in $\out$ has a train track representative on $R_r$).  Moreover, $\mathbf U(\vphi)=\mathbf U(\psi^{-1} \vphi \psi)$ for each $\psi\in \out$. Lemma~\ref{lem:upper-bound} will imply that $\mathbf U(\vphi)$ is finite for every fully irreducible $\vphi\in \out$.
\end{rk}

\begin{rk}\label{rk:count}
Observe that if $\alpha,\beta\in \out$ are both fully irreducible and representable by train track maps on roses, then $\alpha$ is conjugate to $\beta$ in $\out$ if and only if $\mathbf U(\alpha)=\mathbf U(\beta)$, if and only if $\mathbf U(\alpha)\cap \mathbf U(\beta)\ne \varnothing$. Therefore, if we have a collection $S$ of $k\ge 1$ combinatorially distinct train track maps on $r$-roses representing fully irreducible outer automorphisms of $F_r$ and if $m\ge 1$ is such that each $f\in S$ represents $\vphi_f\in\out$ with $\#\mathbf U(\vphi_f)\le m$, then the collection $\{[\vphi_f]\mid f\in S\}$ contains $\ge k/m$ distinct $\out$-conjugacy classes of fully irreducibles.
\end{rk}

For a train track map $f: \G\to \G$ define $||f||:=\sum_{e\in E\G} |f(e)|$, where the summation is taken over all topological edges of $\G$, and where $|f(e)|$ is the combinatorial length of the path $f(e)$.

\begin{lem}\label{lem:multipl}
Let $f: R_r\to R_r$ be a train track map representing an ageometric lone axis fully irreducible $\vphi\in \out$, where $r\ge 3$.
Then
\[
\#\mathbf U(\vphi)\le ||f||.
\]
\end{lem}

\begin{proof}
Choose in the $\out$ conjugacy class of $\vphi$ the outer automorphism $\vphi'\in \out$ that has a train track representative $g$ on the rose $R_r$ with the identity marking.
Thus $g=f$ as a map on the rose $R_r$, and the only difference between $g$ and $f$ is in modifying the marking.

Since having a lone axis is a conjugacy class invariant, $\vphi'$ also has a unique axis $\mA_{\vphi'}$.
By Theorem \ref{t:uniqueaxis}, $\mA_{\vphi'}$ is the periodic fold line obtained from each train track representative of $\vphi'$. Call by $\sigma$ the segment of $\mA_{\vphi'}$ starting at $R_r$ with the identity marking and consisting of a single Stallings fold decomposition of $g$ (a single period of the periodic fold line for $g$). The lone axis property of $\vphi'$, together with the periodicity of the line, imply that all elements of $\mathbf U(\vphi')$ arise from the $r$-roses that occur along $\sigma$.

Using the explanation of a Stallings fold decomposition given in Definition \ref{d:StallingsFoldDecomposition},  we see that a period consists of $||g||-r$ folds. Hence, $\sigma$ passes through at most $||g||-r+1\le ||g||\, $ $r$-roses. (Note that in the middle of a fold the underlying graph always has a trivalent vertex and is therefore never the $r$-rose).  Therefore, the element $\vphi'\in\out$ has at most $||g||$ unmarked train track representatives based on the rose $R_r$, as does its conjugate $\vphi$. Since as unmarked graph maps $g=f$, we have $||g||=||f||$. Hence $\#\mathbf U(\vphi)\le ||f||$, as claimed.
\end{proof}

\begin{thm}\label{thm:lower-bound}
Let $r\ge 3$. Then there exist constants $c=c(r)>1$ and $L_0\ge 1$ such that for each $L\ge L_0$ the number $\mathfrak{NA}_r(L)$ of distinct $\out$-conjugacy classes of ageometric lone axis fully irreducibles $\vphi\in \out$ with $\log \lambda(\vphi)\le L$ satisfies
\[
\mathfrak{NA}_r(L) \ge c^{e^L}.
\]
Therefore, $c^{e^L}$ bounds below the number of equivalence classes of closed geodesics in $\mathcal{M}_r$ of length bounded above by $L$.
\end{thm}

\begin{proof}

Let $L\ge 1$. Let $Z_+(L)$ denote the set of all full positive words of length $e^L$ in $x_2,\dots, x_r$.

By the law of large numbers, the probability that a uniformly at random chosen positive word in $x_2,\dots, x_r$ of length $n$ is full tends to $1$ as $n\to\infty$. Therefore
\[\lim_{n\to\infty} \frac{\#\{w\mid w \text{ is a full positive word  in } x_2,\dots, x_r  \text{ of length } n\}}{(r-1)^{n}}=1\] and so $\lim_{L\to\infty} \frac{\#Z_+(L)}{(r-1)^{e^L}}=1$. In particular, there exists a sufficiently large $L_0'\ge 1$ such that for all $L\ge L_0'$ we have $\frac{\#Z_+(L)}{(r-1)^{e^L}}\ge 1/2$, that is $\#Z_+(L)\ge (r-1)^{e^L}/2$.

For each such word  $z\in Z_+(L)$, the word $w=x_{r-1}zx_{2}$ is also a full positive word and begins in $x_{r-1}$ and ends in $x_2$.
Define $W_+(L):=\{x_{r-1}zx_{2}\mid z\in Z_+(L)\}$. Thus $\#W_+(L) \ge (r-1)^{e^L}/2$ for each $L\ge L_0$, and for each $w\in W_+(L)$ we have $|w|=e^L+2$.

For each $w\in W_+(L)$ consider the train track map $f_w=g_w\circ g_{12,1}$ as in Proposition~\ref{prop:LoneAxes} above. Proposition~\ref{prop:LoneAxes} implies $f_w$ represents an ageometric lone axis fully irreducible $\varphi_w\in \out$.  We claim that for $w,w'\in W_+(L)$, we have $\vphi_w=\vphi_{w'}$ if and only if $w=w'$. Indeed, suppose, on the contrary, that $w,w'\in W_+(L)$ are distinct words, but that $\vphi_w=\vphi_{w'}$. Denote by $\psi_w$, $\psi_{w'}$, and $\beta$ the elements of $\out$ represented by $g_w, g_{w'}$, and $g_{12,1}$ respectively. We have $\vphi_w=\vphi_{w'}=\psi_w\beta=\psi_{w'}\beta$ and therefore $\psi_w=\psi_{w'}=\vphi_w\beta^{-1}=\vphi_{w'}\beta^{-1}$. By definition, $\psi_w([x_r])=[x_1 w(x_2,\dots, x_r)]$ and $\psi_{w'}([x_r])=[x_1 w'(x_2,\dots, x_r)]$. (Here for $u\in F_r$ we denote by $[u]$ the conjugacy class of $u$ in $F_r$). Since by assumption $w\ne w'$ are distinct positive words in $x_2,\dots, x_r$, the words $x_1 w(x_2,\dots, x_r)$ and $x_1 w'(x_2,\dots, x_r)$ are distinct positive words in $x_1,\dots, x_r$, which are cyclically reduced and therefore not conjugate in $F_r=F(x_1,\dots, x_r)$. This contradicts the assumption that $\vphi_w=\vphi_{w'}$. Thus the claim is verified, and we know that distinct words $w\in W_+(L)$ define distinct outer automorphisms $\varphi_w\in \out$. Since for each $w\in W_+(L)$ we have that $f_w:R_r\to R_r$ is a train track representative of $\vphi_w$, it follows that distinct words in $w\in W_+(L)$ produce distinct marked train track maps $f_w:R_r\to R_r$, where $R_r$ is taken with the identity marking. Two such maps $f_w$ and $f_{w'}$ can still be equivalent, in the sense of Definition~\ref{df:U}, if they are conjugate by a graph automorphism of $R_r$. There are $m=2^rr!$ simplicial automorphisms of $R_r$. We thus have, for each $L\ge L_0'$, a collection $\{f_w:R_r\to R_r\mid w\in W_+(L)\}$ of at least $ (r-1)^{e^L}/(2m)$ combinatorially distinct (in the sense of Definition~\ref{df:U}) train track maps.

For $w\in W_+(L)$ we have $||g_w||=r+|w|=r+2+e^L$. Since $g_{12,1}$ is fixed and does not depend on $w$ or $L$, there exists a constant $K=K(r)\ge 1$ such that $||f_w||\le Ke^{L}$.  By Lemma~\ref{lem:multipl}, for each $w\in W_+(L)$, we have $\#\mathbf U(\varphi_w)\le ||f_w||\le Ke^{L}$.

Then by Remark~\ref{rk:count}, for $L\ge L_0'$,  the number of distinct $\out$ conjugacy classes represented by $\{f_w\mid w\in W_+(L)\}$ is
\[
\ge\frac{(r-1)^{e^L}}{2mKe^L}\ge_{L\to\infty} (r-1.5)^{e^L}.
\]

For each $w\in W_+(L)$ we have $||f_w||\le Ke^{L}$, and therefore $\lambda(f_w)\le ||f_w||\le K e^L$. (Here we are using the fact that the Perron-Frobenius eigenvalue $\lambda(f_w)$ is bounded above by the maximum of the row-sums of the transition matrix $M(f_w)$; see~\cite{Seneta} for details.) Hence \[\log\lambda(\varphi_w)=\log \lambda(f_w)\le L+\log K.\]

Now let $L_1=L+\log K$.  Then from above we have $\log\lambda(\varphi_w)=\log \lambda(f_w)\le L+\log K=L_1$. Also, the number of distinct $\out$ conjugacy classes represented by $\{f_w\mid w\in W_+(L)\}$ is
\[
\ge_{L\to \infty} (r-1.5)^{e^L}= (r-1.5)^{e^{L_1-\log K}}= \left((r-1.5)^{1/K}\right)^{e^{L_1}},
\]which completes the proof of the main statement.

The final sentence of the theorem follows from the fact that the translation distance of $\vphi$ along $A_{\vphi}$ is $log(\lambda(\vphi))$.
\end{proof}

\begin{lem}\label{lem:upper-bound}
Let $r\ge 2$. Then there exist $a>1,b>1$ such that for each $L\ge 1$ the number of expanding irreducible train track maps $f:\G\to \G$, where $b_1(\G)=r$ and $\log \lambda(f)\le L$, is $\le a^{b^L}$. 

In particular, for a fully irreducible $\varphi\in\out$, the cardinality of $U(\varphi)$ has a double exponential upper bound in $\log(\lambda(\varphi))$, namely $\#\mathbf U(\phi)\le a^{b^L}$, where $L=\log(\lambda(\varphi))$.
\end{lem}

\begin{proof}
This proof follows the argument in \cite[Remark~3.3]{hk17}.
First note that there are only finitely many choices for a finite connected graph $\G$ satisfying that all vertices have degree $\ge 3$ and $b_1(\G)=r$. Thus we may assume $\G$ is fixed. Let $k=\#E\G$ and let $M=(m_{ij})_{ij=1}^k$ be $M(f)$.
By \cite[Proposition~A.4]{kb16}, if $f$ is as above and $\lambda:=\lambda(f)$, then
$\max m_{ij}\le k \lambda^{k+1}$. If $\log \lambda\le L$, we have $\max \log m_{ij} \le \log k + (k+1) L$ and $\max m_{ij}\le k e^{(k+1)L}$. Thus we obtain exponentially many (in terms of $L$) possibilities for transition matrices $M(f)$. Since for a given length $s$ there are exponentially many paths of length $s$ in $\G$, we obtain a double exponential upper bound for the cardinality of the set $V(L)$ of all expanding irreducible train track maps $f:\G\to \G$ with $\log \lambda(f)\le L$, as required. Since for all fully irreducible $\varphi\in\out$ with $\log(\lambda(\varphi))=L$ we have $\mathbf U(\varphi) \subseteq V(L)$, it follows that the cardinality of $\mathbf U(\varphi)$ has a double exponential bound in $log(\lambda(\varphi))=L$ as well.  
\end{proof}

We can now establish the main result of this paper, Theorem~\ref{thm:main} from the Introduction:

\begin{thm}
For each integer $r\ge 3$, there exist constants $a=a(r)>1,b=b(r)>1, c=c(r)>1$ so that: For $L\ge 1$, let $\mathfrak N_r(L)$ denote the number of $\out$-conjugacy classes of fully irreducibles $\vphi\in \out$ with $\log\lambda(\vphi)\le L$. Then there exists an $L_0\ge 1$  such that for each $L\ge L_0$ we have
\[
c^{e^L}\le \mathfrak N_r(L)\le  a^{b^L}.
\]
Therefore, $c^{e^L}$ bounds below the number of equivalence classes of closed geodesics in $\mathcal{M}_r$ of length bounded above by $L$.
\end{thm}

\begin{proof}
The lower bound follows directly from Theorem~\ref{thm:lower-bound}. Since every fully irreducible $\vphi\in \out$ can be represented by an expanding irreducible  train track map $f:\G\to\G$ with $\lambda (\vphi)=\lambda(f)$, the upper bound follows from Lemma~\ref{lem:upper-bound}.
\end{proof}

\begin{thm}\label{thm:main2'}
Let $r\ge 2$ be an integer.
Then there exists a constant $p=p(r)>1$ such that
\[
\#\{ \lambda(\vphi) | \vphi\in\out \text{ is fully irreducible with } \log\lambda(\vphi)\le L \} \le p^L.
\]
\end{thm}
\begin{proof}
As shown in the proof of Lemma~\ref{lem:upper-bound},  if $f:\G\to\G$ is an expanding irreducible train track map representing some $\vphi\in\out$ with $\log\lambda(\vphi)=\log\lambda(f)\le L$, then for the coefficients $m_{ij}$ of
$M(f)$ we have $\max m_{ij}\le k e^{(k+1)L}$, where $k=\#E\G$. For a fixed $r\ge 2$, since $\pi_1(\G)\cong F_r$ and each vertex of $\G$ has degree $\ge 3$, we have $k\le 3r-3$.  Thus the number $k$ of rows/columns of $M(f)$ is bounded by $3r-3$. Therefore there is a constant $p=p(r)>1$ such that for $L\ge  1$ the number of possible transition matrices $M(f)$ is bounded above by $p^L$.
Hence, for $L\ge 1$,
\[
\#\{ \lambda(\vphi) \mid \vphi\in\out \text{ is fully irreducible with } \log\lambda(\vphi)\le L \} \le p^L.
\]
\end{proof}

\section{Questions}\label{sec:Q}

Define a function $\omega:[0,\infty)\to [0,\infty)$ to have \emph{double exponential asymptotics} if there exist numbers $a>1, b>1, c>1,d>1$ and $t_0\ge 1$ such that for all $t\ge t_0$,
\[
c^{d^t}\le \omega(t) \le a^{b^t}.
\]
Describing the precise asymptotics of such functions appears to be a nontrivial analytic problem. As an initial approach, for a function $\omega(t)$ with  double exponential asymptotics we define the \emph{principal entropy} $b=b(\omega)>1$ as
\[
\log b:=\limsup_{t\to\infty} \frac{\log \log \omega(t)}{t}. \tag{$\dag$}
\]
Note that if $\omega(t)=a^{b^t}$, for constants $a>1,b>1$, then the principal entropy of $\omega$ is exactly $b$.

Now, if  $\omega(t)$ is a function with double exponential asymptotics and with principal entropy $b=b(\omega)$, we define the \emph{secondary entropy} $a=a(\omega)$ as
\[
\log a:=\limsup_{t\to\infty} \frac{\log \omega(\log_b t)}{t}.
\]
Again, if  $\omega(t)=a^{b^t}$, where $a>1,b>1$ are constants, then the secondary entropy of $\omega$ is exactly $a$.

Recall that $\mathfrak N_r(L)$ is the number of $\out$-conjugacy classes of fully irreducibles $\vphi\in \out$ with $\log\lambda(\vphi)\le L$.

\begin{qst}
Let $r\ge 3$ and $\omega(L)=\mathfrak N_r(L)$.
\begin{enumerate}
\item What is the principal entropy $b(\omega)$ of $\omega$?
\item Does $b(\omega)$ depend on $r$?
\item Is it true that $b(\omega)=e$? (Theorem~\ref{thm:lower-bound} does imply that $b(\omega)\ge e$.)
\item Does the actual limit exist for $\omega$ in $(\dag)$ in this case?
\item What is the secondary entropy $a(\omega)$? Does it depend on $r$ and how?
\end{enumerate}
\end{qst}

\vskip1pt

\subsection*{Acknowledgements}

We are grateful to Yael Algom-Kfir, Michael Hull, and Paul Schupp for useful discussions. We also thank Kasra Rafi for letting us know about his result with Bestvina mentioned above. We thank the referee for a careful reading of the paper and for useful suggestions.

\subsection*{Funding}
The first author was supported by the individual NSF
  grants DMS-1710868 and DMS-1905641. The second author was supported by research funds granted via a Ky Fan Visiting Assistant Professorship, a Queen's University Research Initiation Grant, and an NSERC Discovery Grant and Discovery Launch Supplement. Both authors acknowledge support from U.S. National Science Foundation grants DMS 1107452, 1107263, 1107367 "RNMS: GEometric structures And Representation varieties" (the GEAR Network).

\subsection*{Conflicts of interest}

The authors declare that they have no affiliations with or involvement in any organization or entity with any financial interest or non-financial interest in the subject matter or materials discussed in this article.

\subsection*{Authors contribution}

Both authors contributed equally to the research and preparation of this article.

\subsection*{Data availability}

Data sharing is not applicable to this article as no datasets were generated or analyzed during the preparation of this article.

\bibliographystyle{alpha}
\bibliography{References}

\end{document}